\documentclass[10pt]{article}
\usepackage[utf8]{inputenc}
\usepackage[french, english]{babel}
\usepackage[margin=2.5cm]{geometry}
\usepackage[english]{babel}
\usepackage{hyperref,xcolor}
\usepackage{amsmath}
\usepackage{amsthm}
\usepackage{amssymb}
\usepackage{enumitem}
\usepackage{mathtools}
\usepackage{ucs}
\newcommand{\Manoa}{M\=anoa }
\newcommand{\Hawaii}{Hawai\kern.05em`\kern.05em\relax i }
\newtheorem{thm}{Theorem}[section]

\newtheorem{cor}[thm]{Corollary}

\newtheorem{lem}[thm]{Lemma}

\newtheorem{prop}[thm]{Proposition}

\newtheorem{remark}[thm]{Remark}

\let\OLDthebibliography\thebibliography
\renewcommand\thebibliography[1]{
  \OLDthebibliography{#1}
  \setlength{\parskip}{0pt}
  \setlength{\itemsep}{0pt plus 0.3ex}
}
\DeclareMathOperator\Supp{Supp}
\DeclareMathOperator\ssup{sup}
\DeclareMathOperator\iinf{inf}
\DeclareMathOperator\Int{int}
\theoremstyle{definition}

\newtheorem{definition}[thm]{Definition}

\title{Finitely presented simple left-orderable groups in the landscape of Richard Thompson's groups\\
Groupes simples de pr\'esentation finie ordonnables \`a gauche dans le paysage des groupes de Richard Thompson}
\begin{document}
\author{James Hyde and Yash Lodha}

\date{\today}
\maketitle
\maketitle

\begin{abstract}
We construct the first examples of finitely presented simple groups of orientation-preserving homeomorphisms of the real line. Our examples are also of type $F_{\infty}$, have infinite geometric dimension, and admit a nontrivial homogeneous quasimorphism (and hence have infinite commutator width).
\end{abstract}
\begin{otherlanguage}{french}
\begin{abstract}
Nous construisons les premiers exemples de groupes simples de pr\'esentation finie et ordonnables \`a gauche. Nos exemples satisfont \'egalement la propri\'et\'e de finitude plus forte $F_{\infty}$, ont une dimension g\'eom\'etrique infinie, et admettent un quasimorphisme homog\`ene non trivial.
\end{abstract}
\end{otherlanguage}

\section{Introduction.}

A natural line of investigation in modern group theory seeks to understand the landscape of \emph{finitely presented, infinite, simple groups}. The first such examples were discovered by Richard Thompson in $1965$, and are referred to in the literature as Thompson's groups $T$ and $V$ \cite{CFP}. 
In the subsequent decades various ``Thompson-like" examples of finitely presented infinite simple groups emerged in the work of Brown \cite{brown}, Higman \cite{Higman}, Scott \cite{Scott}, Stein \cite{Stein}, Brin \cite{BrinHigher}, Rover-Nekrashevich \cite{SWZ} and more recently, among others, in \cite{LodhaSimple}, \cite{BNR}. A feature of all these examples, including $T$ and $V$,  is that they are not torsion-free. In particular, they cannot act faithfully on the line by orientation-preserving homeomorphisms. 
In their seminal article \cite{BurgerMosesIHES}, Burger and Mozes constructed the first finitely presented simple
torsion-free groups. They obtained a family of such groups, emerging as lattices in products of automorphism groups of regular trees. For each of them it remains unknown whether it admits a nontrivial action by homeomorphisms on the real line.
Finitely presented infinite simple groups were also constructed within the realm of non-affine Kac–Moody groups \cite{CapraceRemy}. However, the groups in \cite{CapraceRemy} are not torsion-free.
In this article we construct the first family of finitely presented left-orderable simple groups, proving:

\begin{thm}\label{thm:main}
There exist finitely presented (and type $\mathbf{F}_{\infty}$) simple groups of orientation-preserving homeomorphisms of $\mathbf{R}$. Equivalently, there exist finitely presented (and type $\mathbf{F}_{\infty}$) simple left-orderable groups.
\end{thm}

Whether finitely generated infinite simple groups of homeomorphisms of $\mathbf{R}$ exist was a longstanding open problem \cite[Problem $16.50$]{Kourovka}.
This was solved by the authors in \cite{HydeLodha}, where we exhibited continuum many (up to isomorphism) examples.
Subsequently, Matte Bon and Triestino in \cite{MBT} provided a more conceptual generalisation of our construction and new classes of examples emerged in the work of the authors with Rivas in \cite{HLR}.
However, none of these examples are finitely presented since they emerge naturally as nontrivial limits in the Grigorchuk space of marked groups. 
Theorem \ref{thm:main} is proved by means of the following new construction.

\begin{definition}\label{ourgroup}



For $n\geq 2$, we define $\Gamma_n\leq \textup{Homeo}^+(\mathbf{R})$ as the group of homeomorphisms $f\in \textup{Homeo}^+(\mathbf{R})$ satisfying the following: 

\begin{enumerate}[itemsep=0pt,parsep=0pt]

\item $f$ is piecewise linear with breakpoints in $\mathbf{Z}[\frac{1}{n(n+1)}]$, and $\mathbf{Z}[\frac{1}{n(n+1)}]\cdot f=\mathbf{Z}[\frac{1}{n(n+1)}]$. (A \emph{breakpoint} is a point where the left and right derivatives do not coincide.)

\item $f$ commutes with the translation $t\mapsto t+1$.

\item For each $x\in \mathbf{R}\setminus \mathbf{Z}[\frac{1}{n(n+1)}]$, there exist (unique) $i,j\in \mathbf{Z}$ such that $x\cdot f'=n^i(n+1)^j$ and one has:
$$i-j=|\mathbf{Z}\cap (x,x\cdot f)|\text{ if } x\leq  x\cdot f\qquad i-j=-|\mathbf{Z}\cap (x\cdot f, x)|\text{ if }x>x\cdot f.$$






\end{enumerate}

\end{definition}


\begin{thm}\label{main}
For each $n\geq 2$, the group $Q_n=[\Gamma_n,\Gamma_n]$ is a finitely presented (and type $\mathbf{F}_{\infty}$) simple group of orientation-preserving homeomorphisms of $\mathbf{R}$.
\end{thm}

{\bf First observations.}
We supply some basic facts about our groups $\Gamma_n,Q_n$.
\begin{prop}\label{firstthoughts}
For each $n\in \mathbf{N},n\geq 2$, let $\Gamma_n$ denote the set of homeomorphisms from Definition \ref{ourgroup} and fix $\eta_n=n(n+1)$.
Then the following holds:
\begin{enumerate}[itemsep=0pt,parsep=0pt]
\item $\Gamma_n$ is a subgroup of $\textup{Homeo}^+(\mathbf{R})$.
\item The stabilizer of $0$ in $\Gamma_n$ is isomorphic to the Higman-Thompson group $F_{\eta_n}$. In particular, $\Gamma_n,Q_n$ have infinite geometric dimension.
\item $\Gamma_n,Q_n$ embed in the group of piecewise linear orientation-preserving homeomorphisms of $\mathbf{S}^1=\mathbf{R}/\mathbf{Z}$. 
\item $\Gamma_n,Q_n$ do not have Kazhdan's property (T), and they contain nonabelian free subgroups.
\end{enumerate}
\end{prop}

We recall the definition of the Higman-Thompson groups $F_n$, for each $n\in \mathbf{N},n\geq 2$. 
$F_n$ is the group of orientation-preserving piecewise linear homeomorphisms $f:[0,1]\to [0,1]$ whose slopes lie in $\{n^m\mid m\in \mathbf{Z}\}$ and breakpoints lie in $\mathbf{Z}[\frac{1}{n}]$.
Denote this as the \emph{standard action} of $F_n$ on $[0,1]$.
We define the \emph{$1$-periodic action} of $F_n$ as the unique embedding $F_n\leq \textup{Homeo}^+(\mathbf{R})$ which commutes with integer translations, fixes $\mathbf{Z}$ pointwise, and whose restriction to $[0,1]$ is the standard action of $F_n$.

\begin{proof}[Proof of Proposition \ref{firstthoughts}]
We will need the following definitions.
For $x,y\in\mathbf{R}$, we define $(x,y)_{\mathbf{Z}}$ as: $$(x,y)_{\mathbf{Z}}:=|(x,y)\cap \mathbf{Z}|\text{ if }x\leq y\qquad (x,y)_{\mathbf{Z}}:=-|(y,x)\cap \mathbf{Z}|\text{ if }y< x$$
This satisfies the relation $(x,y)_{\mathbf{Z}}+(y,z)_{\mathbf{Z}}=(x,z)_{\mathbf{Z}}$ for all $x,y,z\in \mathbf{R}\setminus \mathbf{Z}$.
For each $n\in \mathbf{N},n\geq 2$, fix $\Omega_n:=\{n^{k_1}(n+1)^{k_2}\mid k_1,k_2\in \mathbf{Z}\}$ as the multiplicative subgroup of $\mathbf{R}_{>0}$.
We define the homomorphism $\alpha:\Omega_n\to \mathbf{Z}$ as $\alpha(\gamma)=k_1-k_2$ for $\gamma=n^{k_1}(n+1)^{k_2}\in \Omega_n$.
{\bf Part $(1)$}: $\Gamma_n$ clearly contains the identity.
We show that for $f,g\in \Gamma_n$, $fg\in \Gamma_n$, and leave the proof that $f^{-1}\in \Gamma_n$ as an exercise.
The element $fg$ clearly satisfy conditions $(1),(2)$ of Definition \ref{ourgroup}.
Condition $(3)$ for $fg$ is equivalent to the assertion that for all $x\in \mathbf{R}\setminus \mathbf{Z}[\frac{1}{\eta_n}]$,  $(x,x\cdot fg)_{\mathbf{Z}}=\alpha(x\cdot (fg)')$.
Since $f,g\in \Gamma_n$, we know that $(x,x\cdot f)_{\mathbf{Z}}=\alpha(x\cdot f')$ and $(x,x\cdot g)_{\mathbf{Z}}=\alpha(x\cdot g')$.
The chain rule implies that $\alpha(x\cdot (fg)')=\alpha(x\cdot f')+\alpha((x\cdot f)\cdot g')=(x,x\cdot f)_{\mathbf{Z}}+(x\cdot f,x\cdot fg)_{\mathbf{Z}}=(x,x\cdot fg)_{\mathbf{Z}}$.

{\bf Part $(2)$}: Each element of the $1$-periodic action of $F_{\eta_n}$ satisfies Definition \ref{ourgroup}, hence lies in $\Gamma_n$. 
Any element $f\in \Gamma_n$ that fixes $0$ must pointwise fix $\mathbf{Z}$, due to condition $(2)$ of Definition \ref{ourgroup}. Indeed, it follows from Definition \ref{ourgroup} that $f$ is an element in the $1$-periodic action of $F_{\eta_n}$. 
Finally, it is a standard fact that $F_{\eta_n}'$ contains copies of $\oplus_{\mathbf{Z}}\mathbf{Z}$, which has infinite geometric dimension (a feature inherited by overgroups).

{\bf Parts $(3),(4)$}: By Definition \ref{ourgroup}, the groups $\Gamma_n,Q_n$ have the property that each element commutes with all integer translations, yet no integer translation is contained in $\Gamma_n$.
It follows that both actions descend to faithful actions by orientation-preserving piecewise linear homeomorphisms of $\mathbf{S}^1=\mathbf{R}/\mathbf{Z}$.
It follows from the main theorems in \cite{Cornulier} and \cite{LMBT} that $\Gamma_n$ does not have Kazhdan's property (T). It follows from a standard ``ping-pong argument" for dense subgroups of $\textup{Homeo}^+(\mathbf{S}^1)$ \cite{margulis} (see also Theorem $2.3.2$ in \cite{Navas}) that they contain nonabelian free subgroups.
\end{proof}




{\bf An alternative action.}
A remark of James Belk after the first version of the article appeared provides an elegant alternative action of $\Gamma_2$ on $\mathbf{R}_{> 0}$. 
\begin{remark}\label{Belk}
Consider the piecewise linear, orientation reversing homeomorphism
$\kappa:\mathbf{R}_{> 0}\to \mathbf{R}$ that maps $[2^k,2^{k+1}]\mapsto [-(k+1),-k]$ linearly for each $k\in \mathbf{Z}$.
Note that $\kappa$ conjugates $t\mapsto 2t$ to $t\mapsto t-1$.
In fact, $\kappa^{-1}$ conjugates the given action of $\Gamma_2$ on $\mathbf{R}$ to the group consisting of all piecewise linear maps 
$f\in \textup{Homeo}^+(\mathbf{R}_{> 0})$ satisfying: the slope $x\cdot f'$, whenever defined, lies in $\{6^n\mid n\in \mathbf{Z}\}$, the breakpoints of $f$ lie in $\mathbf{Z}[\frac{1}{6}]\cap \mathbf{R}_{\geq 0}$, and $f$ commutes with $t\mapsto 2t$. 
\end{remark}
An anonymous referee pointed out that we can also define $\Gamma_n, n\geq 2$ in this fashion, as follows.
\begin{remark}
Fix $n\in \mathbf{N}, n\geq 2$, and for each $k\in \mathbf{Z}$, let $u_k=\frac{n^{-k}}{n-1}$. Note that $u_k\in \mathbf{Z}[\frac{1}{n}]+\frac{1}{(n-1)}$ and $u_{k-1}-u_k=n^{-k}$ for all $k\in \mathbf{Z}$.
We define a piecewise linear decreasing map $\kappa:\mathbf{R}\to \mathbf{R}_{>0}$ as follows. For each $k\in \mathbf{Z}$, $k\cdot \kappa=u_k$ and $\kappa$ is affine on each $[k-1,k]$ with slope $-n^{-k}$.
Then $\kappa$ conjugates $\Gamma_n$ to the group of piecewise affine homeomorphisms of the interval $\mathbf{R}_{>0}$ with breakpoints in the additive coset $\mathbf{Z}[\frac{1}{n(n+1)}]+\frac{1}{(n-1)}$, whose slopes are integral powers of $n(n+1)$, and that commute with $x\mapsto nx$.
One may check that for $n=2$ this definition agrees with the action in Remark \ref{Belk}.
\end{remark}

{\bf Applications of the main theorem.}
In their seminal $1984$ article \cite{BrownGeoghegan}, Brown and Geoghegan demonstrated that Thompson's group $F$ is a torsion-free, type $\mathbf{F}_{\infty}$ group with infinite geometric dimension. Their theorem has been considerably generalised \cite{brown}\cite{Stein}\cite{Witzel}, however, no such example was found that is also simple. 
Our main result provides the first example of a type $\mathbf{F}_{\infty}$, torsion-free, simple group with infinite geometric dimension, answering a question of Zaremsky from \cite{ZaremskyOF}. 

Each group $Q_n$ admits a nontrivial homogeneous quasimorphism $Q_n\to \mathbf{R}$ given by $f\mapsto \lim_{m\to \infty} \frac{0\cdot f^m}{m}$.
It follows that $Q_n$ has \emph{infinite commutator width}: for each $m\in \mathbf{N}$, there is an $f\in Q_n$ which cannot be expressed as a product of fewer than $m$ commutators of elements of $Q_n$.
Whether finitely presented simple groups with infinite commutator width exist had been asked as Problem $14.13.(b)$ in the Kourovka notebook \cite{Kourovka}, and the only known solutions before were the aforementioned examples emerging in the work of Caprace and R\'{e}my \cite{CapraceRemy}, as demonstrated in \cite{CapraceFujiwara}. Our groups $\{Q_n\}_{n\geq 2}$ provide the first torsion-free solutions.
\begin{cor}
There exists a finitely presented torsion-free simple group with
an unbounded quasimorphism, and, in particular, with infinite commutator width.
\end{cor}
The reader should also contrast this with the fact that groups in the family of Burger and Mozes \cite{BurgerMosesIHES} do not admit nontrivial homogeneous quasimorphisms \cite{BurgerMonod}; it is unknown whether they have finite commutator width.

Much of the technical portion of this paper is devoted to proving that the groups $\Gamma_n,Q_n$ are of type $\mathbf{F}_{\infty}$. Our main tool in this direction is a special case of a criterion due to Ken Brown (Proposition \ref{brownscriterion}).

\textbf{Acknowledgements:} The authors thank the anonymous referee for their feedback that helped improve the exposition. 
The authors thank Francesco Fournier-Facio for his careful reading of several early drafts of this paper, and thank Matt Brin and James Belk for their important remarks. 
The second author was partially supported by the NSF CAREER award 2240136.

\section{Preliminaries.}

All actions will be right actions.
For a group $G$ and $f,g\in G$, denote $G':=[G,G], G'':=[G',G']$, $f^g:=g^{-1}fg$ and $[f,g]:=f^{-1}g^{-1}fg$. 
For $X\subset G$, denote by $\langle\langle X\rangle \rangle_G$ as the smallest normal subgroup of $G$ containing $X$.
Throughout the paper, we fix $\eta_n=n(n+1)$ for each $n\in \mathbf{N},n\geq 2$.
We shall often identify a given group with some specified action on a $1$-manifold without choosing notation for the action.
The nature of the action will be made clear from the context.


Let $M$ be a connected $1$-manifold, and $G\leq \textup{Homeo}^+(M)$.
Given $f\in G$, define the \emph{support of $f$} as $\Supp(f)=\{x\in M\mid x\cdot f\neq x\}$.
Given $I\subseteq M$, denote 
$\textup{Rstab}_{G}(I)=\{f\in G\mid \Supp(f)\subseteq I\}$.
For $x\in M$, whenever they exist, we denote $x\cdot f',x\cdot f'_-,x\cdot f'_+$ as the derivative, left derivative and right derivative, respectively.
The action of $G$ on $M$ is \emph{minimal} if all orbits are dense in $M$. 
An element $f\in \textup{Homeo}^+(\mathbf{R})$ is said to be \emph{$1$-periodic} if it commutes with all integer translations. A subgroup $G\leq \textup{Homeo}^+(\mathbf{R})$ is said to be $1$-periodic if every element of $G$ is $1$-periodic.

\begin{lem}\label{generalminimal}
Let $G\leq \textup{Homeo}^+(\mathbf{R})$ be a $1$-periodic subgroup such that: 
\begin{enumerate}[itemsep=0pt,parsep=0pt]
\item For each $n\in \mathbf{Z}$, every $x\in (n,n+1)$ and every nonempty open set $U\subset (n,n+1)$, there is an $f\in G$ such that $x\cdot f\in U$.
\item $\mathbf{Z}$ is not $G$-invariant.
\end{enumerate}
Then the action of $G$ on $\mathbf{R}$ is minimal. If such a $G$ does not preserve a Radon measure on $\mathbf{R}$, then the following also holds. For each closed interval $I\subset \mathbf{R}, |I|<1$ and any nonempty open interval $U\subset \mathbf{R}$, there is an $f\in G$
such that $I\cdot f\subset U$.
\end{lem}
\begin{proof}
First we show minimality.
Assume by way of contradiction that there is a proper closed $G$-invariant set $X\subset \mathbf{R}$.
Let $S_1$ be the set of left endpoints of the connected components of $\mathbf{R}\setminus X$ and $S_2$ be the set of right endpoints of these components. Using our hypothesis, since $S_1,S_2$ are $G$-invariant, it is easy to see that $S_1,S_2\subset \mathbf{Z}$.
Then $G$-invariance of $S_1,S_2$ and the $1$-periodicity of $G$ implies that $\mathbf{Z}$ is $G$-invariant, a contradiction.
The second statement follows from a general result (see Theorem $3.5.19$ in \cite{GOD}).
\end{proof}

The finiteness properties type $\mathbf{F}_{n}$ and type $\mathbf{F}_{\infty}$ are of fundamental importance in geometric group theory since they are quasi-isometry invariants of groups \cite{Alonso},\cite{Geoghegan}. 
A group is said to be \emph{of type $\mathbf{F}_n$} if it admits an 
Eilenberg-Maclane complex with a finite $n$-skeleton, and is \emph{of type $\mathbf{F}_{\infty}$} if it is of type $F_n$ for each $n\in \mathbf{N}$.
We recall a special case of what is called Brown's criterion
(Proposition $1.1$ in \cite{brown}).

\begin{prop}\label{brownscriterion}
 Let $\Gamma$ be a group that acts on a cell complex $X$
by cell permuting homeomorphisms such that $X$ is contractible, $X/ \Gamma$ has finitely many cells in each dimension, and the pointwise stabilizer of each cell is of type $\mathbf{F}_{\infty}$.
 Then $\Gamma$ is of type $\mathbf{F}_{\infty}$.
\end{prop}

The following is a standard fact (see \cite{Geoghegan} for a proof).

\begin{prop}\label{extensionfiniteness}
 Consider a group extension $1\to N\to G\to H\to 1$. If $N,H$ are of type $\mathbf{F}_{\infty}$ then $G$ also has type $\textbf{F}_{\infty}$. In particular, a finite direct product of type $\textbf{F}_{\infty}$ groups is also of type $\textbf{F}_{\infty}$.
\end{prop}






Given a group $H$ and an isomorphism $\phi:H\to K$, where $K<H$ is a proper subgroup, the group $\langle H, t\mid t^{-1}ht=\phi(h)\text{ for }h\in H\rangle$ is called an \emph{ascending HNN extension} with \emph{base group} $H$. 
The following is a criterion to verify whether a group admits such a structure (See Lemma $3.1$ in \cite{GMSW} for a proof):
\begin{lem}\label{AHNN}
Let $G$ be a group that satisfies the following. There exist subgroups $H_1< H_2< G$ and an element $f\in G$ such that $f^{-1}H_2f=H_1$, no nontrivial power of $f$ lies in $H_2$,
and $\langle H_2,f\rangle=G$.
Then $G$ admits the structure of an ascending HNN extension with base group $H_2$.
\end{lem}

Finiteness properties of ascending HNN extensions are well behaved (see the end of section $2$ in \cite{BDH}):
\begin{prop}\label{AHNNF}
Let $G$ be an ascending HNN extension with base group $H$. If $H$ has type $\mathbf{F}_{\infty}$, then $G$ has type $\mathbf{F}_{\infty}$.
\end{prop}



\section{The proof.}

\subsection{The groups $F_n,Q_n,\Gamma_n$ and their $1$-periodic actions.}
The Higman-Thompson groups $F_n$ will play an essential role in our proofs.
Recall from Proposition \ref{firstthoughts} that the stabilizer of $0$ in $\Gamma_n$ emerges as the $1$-periodic action of $F_{\eta_n}$ on $\mathbf{R}$.
In \cite{brown}, Brown proved the following. 

\begin{thm}\label{brownthm}
For each $n\geq 2$, the group $F_n$ is of type $\mathbf{F}_{\infty}$.
\end{thm}

In \cite{brown}, Brown isolated the presentation 
$F_n:=\langle(f_i)_{i\in \mathbf{N}}\mid f_i^{-1} f_jf_i=f_{j+n-1}\text{ for }i<j\rangle$
and proved:

\begin{thm}\label{abelianization}
The group $F_n$ is $n$-generated and its abelianization is $\mathbf{Z}^n$. The derived subgroup $F_n'$ is simple and every proper normal subgroup of $F_n$ contains $F_n'$.
\end{thm}

Recall that an action of a group $G$ on a connected $1$-manifold $M$ by homeomorphisms is \emph{proximal} if for every proper compact subset $U\subset M$ and nonempty open subset $V\subset M$,
there is an element $f\in G$ such that $U\cdot f\subset V$.
The following standard facts also emerge in \cite{brown}. 

\begin{lem}\label{Fnproximal}
The standard actions of $F_n,F_n'$ on $(0,1)$ are proximal. Moreover, for each $a,b\in [0,1]\cap \mathbf{Z}[\frac{1}{n}], a<b$, it holds that
$\textup{Rstab}_{F_n}([a,b])\cong F_n$.
\end{lem}

\begin{proof}
The statement concerning proximality was proved in \cite{brown}.
The latter statement follows from observing that for each such $a,b$, there is an $m\in \mathbf{N}$ such that the following holds. 
There is a piecewise linear map $\phi:[a,b]\to [0,m]$ that conjugates $\textup{Rstab}_{F_n}([a,b])$ to the group of piecewise linear homeomorphisms of $[0,m]$ with breakpoints
in $\mathbf{Z}[\frac{1}{n}]\cap [0,m]$ and slopes in $\{n^k\mid k\in \mathbf{Z}\}$. The latter is isomorphic to $F_n$, from Proposition $4.1$ in \cite{brown}.
\end{proof}

Using this, we prove the following.

\begin{prop}\label{ourgroupminimal}
The actions of $\Gamma_n,Q_n$ on $\mathbf{R}$ are minimal. Moreover, for each closed interval $I\subset \mathbf{R}, |I|<1$ and any nonempty open interval $U\subset \mathbf{R}$, there is an $f\in Q_n$
such that $I\cdot f\subset U$.
\end{prop}

\begin{proof}
It suffices to verify the hypothesis of Lemma \ref{generalminimal} for $Q_n$.
By definition, $Q_n$ is $1$-periodic. 
Since the $1$-periodic copy of $F_{\eta_n}'$ is a subgroup of $Q_n$, $Q_n$ satisfies condition $(1)$ of the Lemma \ref{generalminimal} using Lemma \ref{Fnproximal}.
For $I:=[-\frac{1}{\eta_n}, \frac{n}{\eta_n}]$, we define $\lambda:I\to I$ as follows:

$$\lambda:=\begin{cases}
[\frac{-\eta_n}{\eta_n^2},\frac{-\eta_n+1}{\eta_n^2}]\mapsto [\frac{-\eta_n}{\eta_n^2},0] & \text{ linear with slope } \eta_n.\\
[\frac{-\eta_n+1}{\eta_n^2},0]\mapsto [0, \frac{n\eta_n-n}{\eta_n^2}]&\text{ linear with slope }n.\\
[0,\frac{n\eta_n}{\eta_n^2}]\mapsto [\frac{n\eta_n-n}{\eta_n^2},\frac{n\eta_n}{\eta_n^2}]
&\text{ linear with slope }\frac{1}{\eta_n}.
\end{cases}$$




Define the $1$-periodic homeomorphism $f:\mathbf{R}\to \mathbf{R}$ satisfying $\Supp(f)=\Int(I)+\mathbf{Z}$ and $f\restriction I=\lambda$. Clearly, $f\in \Gamma_n$ and $\mathbf{Z}$ is not $f$-invariant. 
If $\mathbf{Z}$ were $Q_n$-invariant, then since $f Q_nf^{-1}=Q_n$ (where $f$ is as above), $\mathbf{Z}\cdot f$ would also be $Q_n$-invariant. This is a contradiction since by Lemma \ref{Fnproximal} the $F_{\eta_n}'$-orbit of $0\cdot f$ is dense in $(0,1)$. 
\end{proof}

Define the map $\theta_n:\mathbf{Z}[\frac{1}{n}]\to \mathbf{Z}/(n-1)\mathbf{Z}$ given by $\frac{k}{n^m}\mapsto k\text{ (mod }n-1)$ for each $k,m\in \mathbf{Z}$.
The following proposition is likely known, but we provide a brief proof for the convenience of the reader.

\begin{prop}\label{orbits}
Fix $n\in \mathbf{N},n\geq 2$. For each $k\geq 1$, we consider two linearly ordered $k$-tuples 
$(x_1,...,x_k)$ and $(y_1,...,y_k)$ in $\mathbf{Z}[\frac{1}{n}]\cap (0,1)$ (here the linear order is induced from the natural one on $\mathbf{R}$).
Then the following are equivalent:
\begin{enumerate}[itemsep=0pt,parsep=0pt]
\item There is an element $f\in F_n$ such that $x_i\cdot f=y_i$ for each $1\leq i\leq k$.
\item There is an element $f\in F_n'$ such that $x_i\cdot f=y_i$ for each $1\leq i\leq k$.
\item $\theta_n(x_i)=\theta_n(y_i)$ for each $1\leq i\leq k$.
\end{enumerate}
It follows that for every $k\in \mathbf{N}\setminus \{0\}$, $F_n'$ has finitely many orbits on $(\mathbf{Z}[\frac{1}{n}]\cap (0,1))^k$.
\end{prop}

We will need the following lemmas.

\begin{lem}\label{lemorbits1}
Every $F_n$-orbit in $\mathbf{Z}[\frac{1}{n}]\cap (0,1)$ meets $W=\{\frac{1}{n},...,\frac{(n-1)}{n}\}$.
\end{lem}
\begin{proof}
Given $x\in \mathbf{Z}[\frac{1}{n}]$, we will construct an element $f\in F_n$ such that $x\cdot f\in W$.
Let $x=\frac{k}{n^{l}}$ and let $k=(n-1)k_1+k_2$ where $k_2\in \{1,...,n-1\}$.
Consider the following ordered partition of $[0,1)$: 
$$\mathcal{I}=\{I_1=\left[0,\frac{1}{n^l}\right),I_2=\left[\frac{1}{n^l},\frac{2}{n^l}\right),...,I_{n^l-1}=\left[\frac{n^l-2}{n^l},\frac{n^l-1}{n^l}\right),I_{n^l}=\left[\frac{n^l-1}{n^l},1\right)\}$$
Given an interval of the form $J=\left[\frac{i}{n^j},\frac{i+1}{n^j}\right)$, we denote by $\Xi(J)$ as the \emph{regular $n$-ary subdivision}:
$$J_1=\left[\frac{in}{n^{j+1}},\frac{in+1}{n^{j+1}}\right),J_2=\left[\frac{in+1}{n^{j+1}},\frac{in+2}{n^{j+1}}\right),...,J_n=\left[\frac{in+n-1}{n^{j+1}},\frac{i+1}{n^j}\right)$$
Applying the operation $\Xi$ to such an interval in a partition of intervals increases the size of the partition by $n-1$.
Build an ordered partition $\mathcal{K}=\{K_1,...,K_{n^l}\}$ of $[0,1)$ from the trivial partition $[0,1)$ as follows. First apply, one by one, the operation $\Xi$ on the leftmost interval $k_1+1$ times
and then on the rightmost interval $\frac{n^{l}-1}{n-1}-k_1-1$ times.
The element $f\in F_n$ that maps each $I_j$ to $K_j$ linearly, maps $x$ in $W$.
\end{proof}

\begin{lem}\label{lemorbits2}
For each $f\in F_n$ and $x\in \mathbf{Z}[\frac{1}{n}]\cap (0,1)$ it holds that $\theta_n(x\cdot f)=\theta_n(x)$.
Moreover, if $f:\mathbf{R}_{\geq 0}\to \mathbf{R}_{\geq 0}$ is a piecewise linear homeomorphism with slopes in 
$\{n^k\mid k\in \mathbf{Z}\}$ and breakpoints in $\mathbf{Z}[\frac{1}{n}]$, then $\theta_n(x\cdot f)=\theta_n(f)$ for each $x\in \mathbf{R}_{\geq 0}$.
\end{lem}

\begin{proof}
Let $x\in \mathbf{Z}[\frac{1}{n}]\cap (0,1)$.
Let $0<t_1<...<t_l<1$ be the set of breakpoints of $f$ and denote $t_0=0,t_{l+1}=1$. 
Let $x\in (t_i,t_{i+1}]$ for some $0\leq i\leq l$.
We show that $\theta_n(x)=\theta_n(x\cdot f)$ by induction on $i$.
For the case $i=0$, this follows since $f\restriction [0,t_1]$ is of the form $t\mapsto n^kt$ for some $k\in \mathbf{Z}$.
For the inductive step, the inductive hypothesis assumes that $\theta_n(t_i)=\theta_n(t_i\cdot f)$.
Since $x\cdot f=n^k(x-t_i)+(t_i\cdot f)$ for some $k\in \mathbf{Z}$, it follows that $\theta_n(x\cdot f)=\theta_n(n^k(x-t_i)+(t_i\cdot f))=\theta_n(n^k(x-t_i))+\theta_n(t_i\cdot f)=\theta_n(x-t_i)+\theta_n(t_i)=\theta_n(x-t_i+t_i)=\theta_n(x)$.
The more general statement follows from a similar inductive argument.
\end{proof}

\begin{lem}\label{lemorbits3}
The following holds:
\begin{enumerate}[itemsep=0pt,parsep=0pt]
\item $F_n$ has $n-1$ orbits on $\mathbf{Z}[\frac{1}{n}]\cap (0,1)$ which are precisely the fibers of $\theta_n\restriction \mathbf{Z}[\frac{1}{n}]\cap (0,1)$.
\item Given $I=(\frac{(n-1)k}{n^l},\frac{(n-1)k+1}{n^l})\subset [0,1]$, 
 $\textup{Rstab}_{F_n}(I)$ has $n-1$ orbits on $\mathbf{Z}[\frac{1}{n}]\cap I$ which are the fibers of $\theta_n\restriction \mathbf{Z}[\frac{1}{n}]\cap I$.
\end{enumerate}
\end{lem}

\begin{proof}
The first part follows immediately from applying Lemmas \ref{lemorbits1}, \ref{lemorbits2}.
We show the second part. The action $\textup{Rstab}_{F_n}(I)$ is topologically conjugate to the standard action of $F_n$ on $(0,1)$ by the homeomorphism that maps $I$ linearly to $[0,1]$.
Moreover, this map is $\theta_n$-equivariant, and hence the assertion in the first part translates naturally to the second part.
\end{proof}



\begin{lem}\label{lemorbits4}
For all $x,y,z\in \mathbf{Z}[\frac{1}{n}]\cap (0,1)$ such that $x,y<z$ and $\theta_n(x)=\theta_n(y)$, there exists $f\in \textup{Rstab}_{F_n}(0,z)$ so that $x\cdot f=y$. 
\end{lem}

\begin{proof}
Choose $k\in \mathbf{N}$ such that $\frac{1}{n^k}<z$.
We find $f_1\in F_n$ such that $x_1=x\cdot f_1,y_1=y\cdot f_1\in [0,\frac{1}{n^k}]$ and $f_1\restriction [z,1]=id$.
From part $(2)$ of Lemma \ref{lemorbits3}, it follows that there is an $f_2\in \textup{Rstab}_{F_n}[0,\frac{1}{n^k}]$ such that $x_1\cdot f_2=y_1$, and hence $f=f_1f_2f_1^{-1}$ is the required element.
\end{proof}

\begin{proof}[Proof of Proposition \ref{orbits}]
The assertion $(1)\implies (3)$ is provided by Lemma \ref{lemorbits2} and $(2)\implies (1)$ is clear.
First, we prove $(3)\implies (1)$ by induction on $k$. The case $k=1$ is part $(1)$ of Lemma \ref{lemorbits3}.
From the inductive hypothesis, we find $f_1\in F_n$ such that $x_i\cdot f_1=y_i$ for $2\leq i\leq k$.
Since $x_1\cdot f_1, y_1\in [0,y_2)$, using Lemma \ref{lemorbits4} we find $f_2\in \textup{Rstab}_{F_n}(0,y_2)$ such that $(x_1\cdot f_1)\cdot f_2=y_1$.
The required element is $f=f_1f_2$.

Now, we prove that $(3)\implies (2)$.
Let $I\subset (0,1)$ be a closed interval containing $\{x_i,y_i\mid 1\leq i\leq k\}$.
Using proximality of the action of $F_n$ on $(0,1)$ (Lemma \ref{Fnproximal}), we find $f_1\in F_n$ such that $I\cdot f_1\in (\frac{n-1}{n^2},\frac{1}{n})$.
Using part $(2)$ of Lemma \ref{lemorbits3} and an argument analogous to the one above, we find $f_2\in \textup{Rstab}_{F_n}(\frac{n-1}{n^2},\frac{1}{n})$ such that $(x_i\cdot f_1)\cdot f_2=(y_i\cdot f_1)$ for each 
$1\leq i\leq k$. 
Now let $J$ be a nonempty open interval in $(0,1)$ such that $(\frac{n-1}{n^2},\frac{1}{n})\cap J=\emptyset$.
Again using proximality, we find $f_3\in F_n$ such that $(\frac{n-1}{n^2},\frac{1}{n})\cdot f_3\subset J$.
The commutator $g=f_3^{-1}f_2^{-1}f_3f_2$ also satisfies that $(x_i\cdot f_1)\cdot g=(y_i\cdot f_1)$ for each 
$1\leq i\leq k$. 
So $f=f_1gf_1^{-1}$ is the required element in $F_n'$ satisfying $x_i\cdot f=y_i$ for each 
$1\leq i\leq k$.
\end{proof}
We will also use the following.
\begin{lem}\label{GammaOrbits}
Let $f\in \Gamma_n,x\in [0,1]$ satisfy $y=x\cdot f\in [0,1]$. Then we have $\theta_{\eta_n}(x)=\theta_{\eta_n}(y)$.
\end{lem}

\begin{proof}
Switching $f$ to $f^{-1}$ and $x$ to $y$ if needed, assume that $z_1=0\cdot f\geq 0$.
Let $z_2\in [0,1]$ be such that $z_2\cdot f=1$.
Note that $1\cdot f=1+z_1$.
Define the homeomorphism $g:\mathbf{R}_{\geq 0}\to \mathbf{R}_{\geq 0}$ as:
$$t\cdot g:=\begin{cases}t\cdot f-z_1\text{ if }t\in [0,z_2]\\
\frac{(t\cdot f)-1}{n}+1-z_1\text{ if }z_2\leq t\leq 1\\
t+\frac{z_1}{n}-z_1\text{ if }t\geq 1
\end{cases}$$
By definition, $g$ is a piecewise linear map $\mathbf{R}_{\geq 0}\to \mathbf{R}_{\geq 0}$ with slopes in $\{(\eta_n)^k\mid k\in \mathbf{Z}\}$ and breakpoints in $\mathbf{Z}[\frac{1}{\eta_n}]$. Applying part $(2)$ of Lemma 
\ref{lemorbits2}, we get $\theta_{\eta_n}(p)=\theta_{\eta_n}(p\cdot g)$ for each $p\in \mathbf{R}_{\geq 0}$.
So, $\theta_{\eta_n}(1)=\theta_{\eta_n}(1\cdot g)= \theta_{\eta_n}(1+\frac{z_1}{n}-z_1)=\theta_{\eta_n}(1)+\theta_{\eta_n}(\frac{z_1}{n})-\theta_{\eta_n}(z_1)$
which implies $\theta_{\eta_n}(\frac{z_1}{n})=\theta_{\eta_n}(z_1)$ and so $\theta_{\eta_n}(nz_1)=\theta_{\eta_n}(z_1)$.
Therefore, $\theta_{\eta_n}((n-1)z_1)=0$.
Since $(n-1), \eta_n-1$ are coprime, it follows that $\theta_{\eta_n}(z_1)=0$.
For $x\in [0,z_2]$, we obtain $\theta_{\eta_n}(x)=\theta_{\eta_n}(x\cdot g)=\theta_{\eta_n}(x\cdot f-z_1)=\theta_{\eta_n}(x\cdot f)-\theta_{\eta_n}(z_1)=\theta_{\eta_n}(x\cdot f)=\theta_{\eta_n}(y)$.
\end{proof}
\subsection{Simplicity.} 

\begin{lem}\label{cornulier}
Let $\Gamma\leq \Gamma_n$ be a subgroup containing the $1$-periodic copy of $F_{\eta_n}'$ and acting minimally on $\mathbf{R}$.
If $N\leq \Gamma$ is a nontrivial normal subgroup, then the following holds:
\begin{enumerate}[itemsep=0pt,parsep=0pt]
\item $N$ contains the $1$-periodic copy of $F_{\eta_n}'$.
\item There is an $f\in N$ such that $0\cdot f\in (0,1)$.
\item $N$ acts minimally on $\mathbf{R}$. 
\item $\Gamma$ is generated by $\Gamma_{<1}=\{g\in \Gamma\mid 0\leq 0\cdot g<1\}$.
\item For each $f\in \Gamma$ such that $0\cdot f\in (0,1)$, there is a $g\in N$ such that $0\cdot fg^{-1}=0$.
\end{enumerate}
\end{lem}

\begin{proof}
{\bf Part $(1)$}. Fix a nontrivial $f\in N$. Then $f$ cannot be an integer translation, so there is an $x\in \mathbf{R}\setminus \mathbf{Z}$ such that $x\cdot f, x-x\cdot f\notin \mathbf{Z}$. We find a small open interval $U,|U|<1$ containing $x$ such that $(U+\mathbf{Z})\cdot f\cap (U+\mathbf{Z})=\emptyset$, $(U\cap \mathbf{Z})=\emptyset$, and $((U\cdot f)\cap \mathbf{Z})=\emptyset$. Let $g\in F_{\eta_n}'\leq \Gamma$ be a nontrivial element with support in $U+\mathbf{Z}$. Then $h=[g,f]$ is nontrivial and pointwise fixes $\mathbf{Z}$, so is in $(N\cap F_{\eta_n})\setminus \{1\}$. Since no nontrivial element of $F_{\eta_n}$ centralizes $F_{\eta_n}'$, there is a $k\in F_{\eta_n}'$ such that $[h,k]\in F_{\eta_n}'\leq \Gamma$ is nontrivial. Since $N$ is normal in $\Gamma$, it follows that $[h,k]\in N\cap F_{\eta_n}'$.
Since $F_{\eta_n}'$ is simple and $F_{\eta_n}'\cap N$ is normal in $F_{\eta_n}'$, it follows that $F_{\eta_n}'\leq N$.

{\bf Part $(2)$}. Using minimality, we choose $f\in \Gamma$ such that $0\cdot f\in (0,1)$. Let $U\subset (0,1)$ be an open interval such that $(U\cdot f^{-1})\subset (0,1)$.
Using the minimality  of $F_{\eta_n}'$ on $(0,1)$, we find $g\in F_{\eta_n}'\leq N$ such that $(0\cdot f)\cdot g\in U$. Then the required element in $N$ is $fgf^{-1}g^{-1}$,
since $0\cdot fgf^{-1}\in (0,1)$ and $(0,1)\cdot g^{-1}=(0,1)$.

{\bf Part $(3)$}. We apply Lemma \ref{generalminimal}, using the previous parts to verify the hypothesis.

{\bf Part $(4)$}. From our hypothesis, $F_{\eta_n}'\leq \Gamma_{<1}$ and there is an $f\in \Gamma_{<1}$ such that $0\cdot f\in (0,1)$.
Applying Lemma \ref{generalminimal}, the group $G$ generated by $\Gamma_{<1}$ acts minimally on $\mathbf{R}$.
So given $h\in \Gamma$, there is a $g\in G$ such that $(0\cdot h)\cdot g\in (0,1)$. Therefore, $hg\in G$ implies that $h\in G$, and hence $G=\Gamma$.

{\bf Part $(5)$}. From part $(2)$, choose $g_1\in N$ such that $0\cdot g_1\in (0,1)$.
By Lemma \ref{GammaOrbits} we know that $\theta_{\eta_n}(0)=\theta_{\eta_n}(0\cdot f)=\theta_{\eta_n}(0\cdot g_1)$. From Proposition \ref{orbits}, there is a $g_2\in F_{\eta_n}'\leq N$ satisfying that $0\cdot g_1=(0\cdot f)\cdot g_2$. It follows that $(0\cdot f)\cdot  g_2g_1^{-1}=0$ and hence $g=g_2g_1^{-1}$ is the required element.
\end{proof}






\begin{lem}\label{properquotient}
Let $\Gamma\leq \Gamma_n$ be a subgroup containing the $1$-periodic copy of $F_{\eta_n}'$ and acting minimally on $\mathbf{R}$. 
Then every proper quotient of $\Gamma$ is abelian. In particular $\Gamma'=\Gamma''$.
\end{lem}

\begin{proof}
Let $N\leq \Gamma$ be a normal subgroup and $\phi:\Gamma\to \Gamma/N$.
From Lemma \ref{cornulier} part $(1)$, $F_{\eta_n}'\leq N$, and so from Theorem \ref{abelianization} $\phi(F_{\eta_n})$ is abelian. So the proof reduces to the claim: $\phi(\Gamma)=\phi(F_{\eta_n}\cap \Gamma)$.
From part $(4)$ of Lemma \ref{cornulier}, $\Gamma_{<1}=\{g\in \Gamma\mid 0\leq 0\cdot g<1\}$ generates $\Gamma$.
Let $f\in \Gamma_{<1}$.
Using part $(5)$ of Lemma \ref{cornulier}, find $g\in N$ such that $(0\cdot f)\cdot g=0$.
It follows that $\phi(f)=\phi(fg)$ and $fg\in F_{\eta_n}$. So our claim follows.

Since $F_{\eta_n}'$ is not metabelian, $\Gamma$ is not metabelian,
i.e., $\Gamma/\Gamma''$ is a proper quotient of $\Gamma$. By the first assertion, $\Gamma/\Gamma''$ is abelian. This means that $\Gamma'=\Gamma''$.
\end{proof}

\begin{prop}\label{simplicity}
The group $Q_n$ is simple.
\end{prop}

\begin{proof}
From Lemma \ref{properquotient} it follows that $\Gamma_n/[Q_n,Q_n]$ is abelian, implying that $Q_n=[Q_n,Q_n]$.
Let $N\leq Q_n$ be a nontrivial normal subgroup. Again, applying Lemma \ref{properquotient} (using the minimality of $Q_n$ on $\mathbf{R}$ from Lemma \ref{ourgroupminimal}), we obtain that $Q_n/N$ is abelian, which must be trivial since $Q_n=[Q_n,Q_n]$.
\end{proof}

\subsection{Finiteness properties.}

For $I\subseteq [0,1]$ a closed interval and $G\leq \textup{Homeo}^+[0,1]$, define $\textup{Rstab}_G^c(I)=\{f\in \textup{Rstab}_G(I)\mid \iinf(I)\cdot f'_+=\ssup(I)\cdot f'_-=1\}$.
Given a $1$-periodic group $G\leq \textup{Homeo}^+(\mathbf{R})$ and an interval $I\subset \mathbf{R}, |I|\leq 1$, 
define $\Upsilon_{G}(I)=\textup{RStab}_G(\Int(I)+\mathbf{Z})$ and $\Upsilon_{G}^c(I)=\{f\in \Upsilon_{G}(I)\mid \iinf(I)\cdot f'_+=\sup(I)\cdot f'_-=1\}$.

Let $G$ be a group and $N$ a subgroup such that $N$ is of type $\mathbf{F}_{\infty}$, $G$ is finitely generated, and $G'\leq N$.
Let $H$ satisfy that $N\leq H\leq G$. Then the pair $(N,G)$ is said to form a \emph{casing pair} for $H$.

\begin{lem}\label{casing}
If a group $H$ admits a casing pair $N,G$, then $H$ is of type $\mathbf{F}_{\infty}$.
\end{lem}
\begin{proof}
$N$ is normal in $G$ since $G'\leq N$.
Consider the extension $1\to N\to H\to H/N\to 1$.
Since $G'\leq N$ and $G$ is finitely generated, $H/N$ is a subgroup of the finitely generated abelian group $G/N$, hence is of type $\mathbf{F}_{\infty}$.
Since $N$ is also of type $\mathbf{F}_{\infty}$, from Proposition \ref{extensionfiniteness} it follows that $H$ is of type $\mathbf{F}_{\infty}$.
\end{proof}

\begin{lem}\label{FPlem1}
For each $a,b\in \mathbf{Z}[\frac{1}{\eta_n}], 0<b-a\leq 1$, there exists $f\in \Upsilon_{Q_n}([a,b])$ such that $a\cdot f_+',b\cdot f_-'=\eta_n$.
\end{lem}

\begin{proof}
First we consider the case when $b-a<1$. 
From Proposition \ref{ourgroupminimal}, there is an $f_1\in Q_n$ such that $(a,b)\cdot f_1\subset (0,1)$.
Using Proposition \ref{orbits}, we find $f_2\in F_{\eta_n}$ such that $a\cdot f_1 < (a\cdot f_1)\cdot f_2< b\cdot f_1< (b\cdot f_1)\cdot f_2$.
We find $h\in F_{\eta_n}$ with support in $(a\cdot f_1, (b\cdot f_1)\cdot f_2^{-1})+\mathbf{Z}$ whose right slope at $a\cdot f_1$ is $\eta_n$ and left slope at $(b\cdot f_1)\cdot f_2^{-1}$
is $\frac{1}{\eta_n}$. Thus $f=f_1(f_2^{-1}h^{-1} f_2h)f_1^{-1}=[f_2,h]^{f_1^{-1}}$ is the required element.
Next, assume the case $b=a+1$. Let $c\in (a,b)\cap \mathbf{Z}[\frac{1}{\eta_n}]$ and find such elements $f_1,f_2$ for $[a,c],[c,b]$, respectively, using the above.
It follows that $f=f_1f_2$ is the required element.
\end{proof}

\begin{lem}\label{FPlem2}
Consider the standard action of $F_{n}$ on $[0,1]$.
Let $I, J\subset (0,1)$ be closed intervals with endpoints in $\mathbf{Z}[\frac{1}{n}]$ such that $J\subset \Int(I)$. 
Then $\textup{Rstab}_{F_n}(I)'\cap \textup{Rstab}_{F_n}^c(J)=\textup{Rstab}_{F_n}(J)'$.
\end{lem}

\begin{proof}
Indeed, $\textup{Rstab}_{F_n}(J)'\subseteq \textup{Rstab}_{F_n}(I)'\cap \textup{Rstab}_{F_n}^c(J)$ is immediate. We show the reverse inclusion.
Let $I=[a,b]$, $g_1,g_2\in \textup{Rstab}_{F_n}(I)$ be such that $g=[g_1,g_2]\in \textup{Rstab}_{F_n}(I)'\cap \textup{Rstab}_{F_n}^c(J)$. Let $a_1,b_1\subset \Int(J)\cap \mathbf{Z}[\frac{1}{n}]$ be such that $\Supp(g)\subset (a_1,b_1)$.
From Proposition \ref{orbits} we find $h_1\in F_n$ so that $\inf(J)<a\cdot h_1<a_1<b_1<b\cdot h_1<\sup(J)$ and $a_1\cdot h_1=a_1,b_1\cdot h_1=b_1$.
Let $h_2\in F_n$ so that $h_2\restriction [a_1,b_1]=h_1\restriction [a_1,b_1]$ and $\Supp(h_2)\subseteq [a_1,b_1]$. 
Then $h=h_1h_2^{-1}$ satisfies that $x\cdot h=x \text{ for each }x\in (a_1,b_1)$, and $\inf(J)<a\cdot h<a_1<b_1<b\cdot h<sup(J)$.
We have $[g,h]=1$, and hence $g=g^h=[g_1^h,g_2^h]\in \textup{Rstab}_{F_n}(J)'$.
\end{proof}


\begin{lem}\label{newing}
For each $a,b\in \mathbf{Z}[\frac{1}{\eta_n}], 0<b-a\leq 1$, $\Upsilon_{\Gamma_n}([a,b])$ is finitely generated and $\Upsilon_{\Gamma_n}([a,b])'=\Upsilon_{\Gamma_n}([a,b])''$.
\end{lem}

\begin{proof}
If $b-a<1$, then from Lemma \ref{ourgroupminimal} there is an $f\in \Gamma_n$ such that $I=[a,b]\cdot f\subset (0,1)$.
It follows that $\Upsilon_{\Gamma_n}([a,b])\cong \Upsilon_{\Gamma_n}(I)=\Upsilon_{F_{\eta_n}}(I)\cong F_{\eta_n}$, where the last identification follows from Lemma \ref{Fnproximal}.
Our conclusion follows.
Now assume that $b=a+1$, and without loss of generality assume $b\in (0,1]$.
From Lemma \ref{FPlem1}, find $f\in \Upsilon_{\Gamma_n}([a,b])$ such that $a\cdot f_+'=\eta_n$.
Let $c\in (a,b)\cap \mathbf{Z}[\frac{1}{\eta_n}]$ be such that $x\cdot f>x, \forall x\in (a,c]$.
From Lemma \ref{AHNN}, $\Upsilon_{\Gamma_n}([a,b])=\langle f, \Upsilon_{\Gamma_n}([c,b])\rangle $ is an ascending HNN extension with base group $\Upsilon_{\Gamma_n}([c,b])\cong F_{\eta_n}$ (from the above, since $b-c<1$), hence is of type $\mathbf{F}_{\infty}$ from Proposition \ref{AHNNF}. 

Let $I=[a,b]$. We show that $\Upsilon_{\Gamma_n}(I)''=\Upsilon_{\Gamma_n}(I)'$ by showing the quotient $\phi: \Upsilon_{\Gamma_n}(I)\to \Upsilon_{\Gamma_n}(I)/\Upsilon_{\Gamma_n}(I)''$ is abelian.
Choose $k_1\in \mathbf{Z}[\frac{1}{\eta_n}]\cap (a,b)$.
Using Lemma \ref{FPlem1}, find $f\in \Upsilon_{\Gamma_n}([a,k_1])$ so that $a\cdot f_+'=\eta_n$.
Therefore, $\Upsilon_{\Gamma_n}(I)=\langle f,G\rangle$, where $G= \{g\in \Upsilon_{\Gamma_n}(I)\mid a\cdot g_+'=1\}$.
Let $k_2\in \mathbf{Z}[\frac{1}{\eta_n}]\cap (k_1,b)$. 
For $\gamma\in G$ we can find $\gamma_1\in \Upsilon_{\Gamma_n}(I)''$ so that $\gamma_1^{-1}\gamma \gamma_1\in \Upsilon_{\Gamma_n}([k_1,b])$.
It follows that $\phi(\Upsilon_{\Gamma_n}(I))=\phi(\langle f\rangle \oplus \Upsilon_{\Gamma_n}([k_1,b]))$.
Since every proper quotient of $\Upsilon_{\Gamma_n}([k_1,b])\cong F_{\eta_n}$ (see the proof of part $(1)$) is abelian, $\phi(\Upsilon_{\Gamma_n}(I))$ is abelian.
\end{proof}


\begin{lem}\label{FPlem3}
For all $\Gamma$ satisfying $Q_n\leq \Gamma\leq \Gamma_n$ and $I=[a,b]\subset \mathbf{R}$ with $|I|\leq 1$, 
$\Upsilon_{\Gamma}(I)$ is of type $\mathbf{F}_{\infty}$.
\end{lem}

\begin{proof}
Assume without loss of generality that $b\in (0,1]$.
Using Lemma \ref{FPlem1}, we find $f\in \Upsilon_{Q_n}(I)$ such that $a\cdot f_+',b\cdot f_-'>1$.
Using the second part of Proposition \ref{ourgroupminimal}, upon replacing $I$ by $I\cdot g$ (and $f$ by $f^g$) for some $g\in Q_n$ if needed, we can assume the following.
First, if $|I|<1$, then $I\subset (0,1)$. Moreover, there exists a closed interval $J\subset (0,1)\cap (a,b)$ with endpoints in $\mathbf{Z}[\frac{1}{\eta_n}]$ so that
$J\cdot f\subset J$ and $\bigcup_{n\in \mathbf{N}}J\cdot f^{-n}=(a,b)$.

We will prove the lemma by isolating a group $H\cong F_{\eta_n}$ such that $\langle f,H\rangle, \Upsilon_{\Gamma_n}(I)$ is 
a casing pair for $\Upsilon_{\Gamma}(I)$, which will imply that $\Upsilon_{\Gamma}(I)$ is of type $\mathbf{F}_{\infty}$ using Lemma \ref{casing}.

Let $J_0=J\cdot f$. We choose $s_i\in \Upsilon_{F_{\eta_n}'}(J)\leq \Upsilon_{\Gamma}(I)$ such that $\{J_i=J_0\cdot s_i\mid 1\leq i\leq \eta_{n}$\} 
are pairwise disjoint intervals in $J\setminus J_0$.
We know that $\Upsilon_{\Gamma_n}(J_0)\cong \textup{Rstab}_{F_{\eta_n}}(J_0)\cong F_{\eta_n}$ (the latter follows from Lemma \ref{Fnproximal}). 
Fix a generating set $u_1,...,u_{\eta_n}$ for $\Upsilon_{\Gamma_n}(J_0)$ and set $v_i=u_is_i^{-1}u_i^{-1}s_i$. Since the set of relators of $F_{\eta_n}$ in the prescribed generating set $u_1,...,u_{\eta_n}$ are all 
products of commutators, $H=\langle v_1,...,v_{\eta_n}\rangle$ also satisfies them and hence is isomorphic to $F_{\eta_n}$.
It follows that $H'=\Upsilon_{\Gamma_n}(J_0)'$.
Moreover, by definition, $H\leq \Upsilon_{\Gamma_n}(J)'\leq Q_n\leq \Gamma$, and so $\langle f, H\rangle\leq \Upsilon_{\Gamma}(I)$.

First, note that $f^{-1}Hf\subset H$, since each $f^{-1}v_if\in \Upsilon_{\Gamma_n}(J_0)'=H'$.
So by Lemma \ref{AHNN}, $\langle f,H\rangle$ is an ascending HNN extension of $H$, and hence of type $\mathbf{F}_{\infty}$ by Proposition \ref{AHNNF}.
Also, $\Upsilon_{\Gamma_n}(I)$ is of type $\mathbf{F}_{\infty}$ from Lemma \ref{newing}.
It remains to show that $\Upsilon_{\Gamma_n}(I)'\leq \langle f, H\rangle$, which reduces to:

{\bf Claim}: $\Upsilon_{\Gamma_n}(I)'\subseteq \bigcup_{n\in \mathbf{N}}f^{n}Hf^{-n}$.

{\bf Case $1$}: $|I|<1$. In this case, we assumed at the beginning that $I\subset (0,1)$.
Therefore, from Lemma \ref{FPlem2}, $\Upsilon_{\Gamma_n}(I)'\cap \Upsilon_{\Gamma_n}^c(J_0)=\Upsilon_{\Gamma_n}(J_0)'$, since $\Upsilon_{\Gamma_n}(L)=\textup{Rstab}_{F_{\eta_n}}(L+\mathbf{Z})$ for an interval $L\subset [0,1]$.
Indeed, for any $k\in \Upsilon_{\Gamma_n}(I)'$, there is an $n\in \mathbf{N}$ such that $f^{-n}kf^{n}\in \Upsilon_{\Gamma_n}^c(J_0)$, and so $f^{-n}kf^{n}\in \Upsilon_{\Gamma_n}(J_0)'=H'$.

{\bf Case $2$}: $|I|=1$. 
From Lemma \ref{newing}, $\Upsilon_{\Gamma_n}(I)''=\Upsilon_{\Gamma_n}(I)'$, ensuring that each element of $\Upsilon_{\Gamma_n}(I)'$ lies in 
$\Upsilon_{\Gamma_n}(I_1)'$ for some interval $I_1\subset I, |I_1|<1$.
The claim that is then proved by applying the same proof as for $|I|<1$ to each such $I_1$, one by one.
\end{proof}

For each $\Gamma$ such that $Q_n\leq \Gamma\leq \Gamma_n$ and $K\subset \mathbf{Z}[\frac{1}{\eta_n}]$, 
define $\Gamma_K=\{f\in \Gamma\mid k\cdot f=k, \forall k\in K\}$. For $K\subset \mathbf{Z}[\frac{1}{\eta_n}]/\mathbf{Z}$ and for the induced circle action of $\Gamma$ on $\mathbf{R}/\mathbf{Z}$,
define $\Gamma_{K,1}=\{f\in \Gamma\mid k\cdot f=k, \forall k\in K\}$.
(Equivalently, for $K\subset \mathbf{Z}[\frac{1}{\eta_n}]/\mathbf{Z}$, $\Gamma_{K,1}=\{f\in \Gamma\mid (k+\mathbf{Z})\cdot f=k+\mathbf{Z},\forall k\in K\}$ for the action on $\mathbf{R}$ .)

\begin{prop}\label{FPProp1}
For $\Gamma$ so that $Q_n\leq \Gamma\leq \Gamma_n$ and every nonempty finite set $K\subset \mathbf{Z}[\frac{1}{\eta_n}]$, $\Gamma_K=\Gamma_{K+\mathbf{Z}}$ is of type $\mathbf{F}_{\infty}$.
\end{prop}

\begin{proof}
Since $\Gamma$ is $1$-periodic, $\Gamma_{K}=\Gamma_{K+\mathbf{Z}}$. 
Using this, changing $K$ if needed and keeping $\Gamma_K$ fixed, suppose that $max(K)<min(K)+1$.
Thus, this provides $|K|$ 
(left closed, right open) intervals $L_1,...,L_{|K|}$ such that $L=\bigcup_{1\leq i\leq |K|}L_i$ is an interval whose closure has length $1$
and $L+\mathbf{Z}$ partitions $\mathbf{R}$.
Let $R=\prod_{1\leq i\leq |K|}\textup{Rstab}_{\Gamma_n}(L_j)$ and $R_1=\prod_{1\leq i\leq |K|}\textup{Rstab}_{\Gamma}(L_j)$.
From our hypothesis, it follows that $R'\subseteq R_1\subseteq \Gamma_K\subseteq R$.
Now $R_1,R$ are of type $\mathbf{F}_{\infty}$, from Lemma \ref{FPlem3} and Proposition \ref{extensionfiniteness}.
So they form a casing pair for $\Gamma_K$, which is henceforth of type $\mathbf{F}_{\infty}$ from Lemma \ref{casing}.
\end{proof}

\begin{thm}\label{FPmain}
Each $\Gamma$ satisfying $Q_n\leq \Gamma\leq \Gamma_n$ is of type $\mathbf{F}_{\infty}$.
\end{thm}

\begin{proof}
We consider the actions of $\Gamma, \Gamma_n$ on $X=\mathbf{Z}[\frac{1}{\eta_n}]/\mathbf{Z}$,
and the natural extension of this action to that on the simplicial complex whose $k$-simplices are the $(k+1)$-element subsets of $X$.
For each nonempty finite set $K\subset X$, $\Gamma_{K,1}$ is an extension of $\Gamma_{K+\mathbf{Z}}$ by a cyclic (trivial or infinite cyclic) group.
Since from Proposition \ref{FPProp1} $\Gamma_{K+\mathbf{Z}}$ is of type $\mathbf{F}_{\infty}$, it follows from applying Proposition \ref{extensionfiniteness} that $\Gamma_{K,1}$ is of type $\mathbf{F}_{\infty}$.
From Proposition \ref{orbits}, for each $k\in \mathbf{N}\setminus \{0\}$, the action of $F_{\eta_n}'$ on $(\mathbf{Z}[\frac{1}{\eta_n}])^k$ has finitely many orbits. Hence the same holds for $\Gamma$ since $F_{\eta_n}'\leq Q_n\leq \Gamma$.
Using Theorem \ref{brownscriterion}, we conclude that $\Gamma$ is of type $\mathbf{F}_{\infty}$.
\end{proof}

\begin{proof}[Proof of Theorem \ref{main}]
$Q_n$ is simple from Proposition \ref{simplicity}, and of type $\mathbf{F}_{\infty}$ from Theorem \ref{FPmain}.
\end{proof}




\bibliographystyle{abbrv}
\bibliography{references}

\noindent{\textsc{Department of Mathematics,
University of Copenhagen.}}

\noindent{\textit{E-mail address:} \texttt{jameshydemaths@gmail.com }}\\

\noindent{\textsc{Department of Mathematics,
University of \Hawaii at \Manoa.}}

\noindent{\textit{E-mail address:} \texttt{lodha@hawaii.edu}} \\

\end{document}